\documentclass[final,leqno]{siamltex}

\usepackage{graphicx}
\usepackage{amsmath}

\usepackage{multirow}

\newtheorem{example}{Example}[section]
\newtheorem{algorithm}{Algorithm}[section]
\newcommand{\alglist}{
\begin{list}{Step 1}
{\setlength{\leftmargin}{0.6 in}\setlength{\labelwidth}{1.0 in}} }

\title{Tensor absolute value equations%\thanks{}
}

\author{Shouqiang~Du\thanks{College of Mathematics and
statistic, Qingdao University, Qingdao 266071, P. R. China ({\tt
dsq8933@163.com}).} \and Liping~Zhang\thanks{Corresponding author.
Department of Mathematical Sciences, Tsinghua University, Beijing
100084, P. R. China ({\tt lzhang@math.tsinghua.edu.cn}).} \and
Chiyu~Chen\thanks{Department of Mathematical Sciences, Tsinghua
University, Beijing 100084, P. R. China ({\tt
ccy\_justin@163.com}).}
        \and Liqun~Qi\thanks{Department of Applied Mathematics, The Hong Kong Polytechnic University, Hung Hom, Kowloon,
Hong Kong, P. R. China ({\tt maqilq@polyu.edu.hk}).} }

\begin{document}

\maketitle

\begin{abstract}
This paper is concerned with solving some structured multi-linear
systems, which are called tensor absolute value equations. This kind
of absolute value equations is closely related to tensor
complementarity problems and is a generalization of the well-known
absolute value equations in the matrix case. We prove that tensor
absolute value equations are equivalent to some special structured
tensor complementary problems. Some sufficient conditions are given
to guarantee the existence of solutions for tensor absolute value
equations. We also propose a Levenberg-Marquardt-type algorithm for
solving some given tensor absolute value equations and preliminary
numerical results are reported to indicate the efficiency of the
proposed algorithm.
\end{abstract}



\pagestyle{myheadings} \thispagestyle{plain} \markboth{S. DU, L.
ZHANG, C. CHEN, AND L. QI}{TENSOR ABSOLUTE VALUE EQUATIONS}

\section{Introduction}
The systems of multi-linear equations can be expressed by
tensor-vector products, just as we rewrite linear systems by
matrix-vector products. Let $\mathcal{A}$ be an $m$th-order tensor
in $R^n\times \cdots \times R^n$ and ${\bf b}$ be a vector in $R^n$.
Then a multi-linear equation can be expressed as
\begin{equation}\label{mequ}
\mathcal{A}{\bf x}^{m-1}={\bf b},
\end{equation}
where $\mathcal{A}{\bf x}^{m-1}$ is a vector in $R^n$ \cite{qi05}
with
$$
(\mathcal{A}{\bf x}^{m-1})_i=\sum_{i_2=1}^n\cdots
\sum_{i_m=1}^na_{ii_2\ldots i_m}x_{i_2}\cdots x_{i_m},\quad
i=1,\ldots,n.
$$
Solving multi-linear systems is always an important problem in
engineering and scientific computing \cite{dingwei,ling}. In this
paper, we consider the systems of multi-linear absolute value
equations, which can be expressed as
\begin{equation}\label{TAVE}
\mathcal{A}{\bf x}^{m-1}-|{\bf x}|^{[m-1]}={\bf b},
\end{equation}
where $|{\bf x}|^{[m-1]}$ is a vector in $R^n$ with
$$
|{\bf x}|^{[m-1]}=(|x_1|^{m-1},\ldots,|x_n|^{m-1})^T.
$$

It is easy to see that the system of multi-linear absolute value
equations (\ref{TAVE}) is a generalization of the well-known
absolute value equations $$A{\bf x}-|{\bf x}|={\bf b}$$ with a
matrix $A\in R^{n\times n}$. The absolute value equations (AVE) has
wide applications in applied science and technology such as
optimization physical and economic equilibrium problems
\cite{mangas1,mangas2,mangas3}. As was shown in \cite{mangas3}, the
general NP-hard linear complementarity problem \cite{LCP} which
subsumes many mathematical programming problems can be formulated as
an AVE. This implies that the AVE is NP-hard in its general form.
Analogous to AVE, we call (\ref{TAVE}) tensor absolute value
equations (TAVE). Obviously, the TAVE is also NP-hard. Thus,
investigating the existence of solutions for the TAVE is a
significant problem.

Recently, Song and Qi \cite{songqi} introduced a class of
complementarity problems, called tensor complementarity problems,
where the involved function is defined by some homogenous polynomial
of degree $n$ with $n > 2$. It is known that the tensor
complementarity problem is a generalization of the linear
complementarity problem \cite{LCP}; and a subclass of nonlinear
complementarity problems \cite{NCP}. The tensor complementarity
problem was studied recently by many scholars
\cite{cheqiwei,songqi2}. In \cite{mangas3}, it was shown that the
AVE is equivalent to a generalized linear complementarity problem.
Can we show that the TAVE is equivalent to a generalized tensor
complementarity problem? Although some computational methods have
been presented for the AVE, it is very difficult to extend these
algorithms to solve the TAVE because the TAVE (\ref{TAVE}) is a
nonlinear equation. The Levenberg-Marquardt method is one of the
important algorithms for solving nonlinear equations \cite{FK}. Can
we propose an efficient algorithm such as the Levenberg-Marquardt
method for solving the TAVE (\ref{TAVE})? To our best knowledge,
there is no general answer to these questions. Therefore, we shall
focus on some special tensor absolute value equations.

Let $\mathcal{I}$ be an $m$th-order $n$-dimensional unit tensor,
whose entries are $1$ if and only if $i_1=\cdots =i_m$ and otherwise
zero.  A tensor $\mathcal{A}$ is called a nonnegative tensor if all
its entries are nonnegative, denoted $\mathcal{A}\ge 0$. A tensor
 is called a $Z$-tensor, if all its diagonal entries are
nonnegative and off-diagonal entries are nonpositive. $M$-tensor is
a special class of Z-tensor, which was first introduced and studied
in \cite{dingqiwei,zhangqizhou}. To define the $M$-tensors, we need
to introduce the tensor eigenvalues first. Let $\mathcal{A}$ be an
$m$th-order $n$-dimensional tensor. If a scalar $\lambda \in R$  and
a nonzero vector ${\bf x}\in R^n$ satisfy
$$
\mathcal{A}{\bf x}^{m-1}=\lambda {\bf x}^{[m-1]},
$$
where ${\bf x}^{[m-1]}=(x_1^{m-1},\ldots,x_n^{m-1})^T$, then we call
$\lambda$ an eigenvalue of $\mathcal{A}$ and $x$ a corresponding
eigenvector. Qi \cite{qi05} and Lim \cite{lim} first defined the
eigenvalues of tensors independently. The spectral radius of a
tensor $\mathcal{A}$ is defined by
$$
\rho(\mathcal{A})=\max\{|\lambda|:\,  \mbox{ $\lambda$ is an
eiegnvalue of $\mathcal{A}$}\}.
$$
A tensor $\mathcal{A}$ is called an $M$-tensor, if it can be written
as $\mathcal{A} = s\mathcal{I}-\mathcal{B}$ with $\mathcal{B}\ge 0$
 and $s\ge \rho(\mathcal{B})$; furthermore, it is called a
strong $M$-tensor if $s> \rho(\mathcal{B})$. One can refer to a
survey \cite{ZGQ} for the spectral theory of nonnegative tensors. In
this paper, we first investigate the existence of solutions for the
TAVE (\ref{TAVE}). We show that the TAVE (\ref{TAVE}) with positive
right-hand side $b$ always has a unique solution when
$\mathcal{A}-\mathcal{I}$ is strong $M$-tensor. Another sufficient
condition for the  existence of solution is also given.  Can we
compute the solution? We propose an inexact Levenberg-Marquardt
method for solving the TAVE (\ref{TAVE}).

The rest of this paper is organized as follows. In Section 2,  we
introduce the tensor absolute value equations which is a
generalization of absolute value equations with matrix case. In
Section 3, some sufficient conditions for the existence of solution
to the TAVE are given. In Section 4, we first reformulate the TAVE
as a special tensor complementarity problem and then we propose an
an inexact Levenberg-Marquardt-type algorithm for solving the TAVE.
Some numerical results are reported in Section 5. Finally, some
conclusions are given.

Throughout this paper, we assume that $m\ge 2$. We use small letters
$\texttt{x, y,\ldots,}$ for scalars, small bold letters ${\bf x},
{\bf y},\ldots,$ for vectors, capital letters $A, B,\ldots,$ for
matrixes, calligraphic letters $\mathcal{A}, \mathcal{B}, \ldots,$
for tensors, calligraphic letters $\mathcal{D}$ for diagonal tensor
whose diagonal elements are $1$ or $-1$. All the tensors discussed
in this paper are real. $T(m,n)$ denotes the set of all $m$th order
$n$-dimensional tensors. Let $\mathcal{A}=(a_{i_1i_2\ldots i_m})\in
T(m,n)$, then $\mathcal{A}$ is called a symmetric tensor if its
entries $a_{i_1i_2\ldots i_m}$ are invariant under any permutation
of their indices. $S(m,n)$ denotes the set  of all  symmetric
tensors. For such a matrix $A^T$ will denote the transpose of $A$.
The identity matrix of arbitrary dimension will be denoted by $I$.

\section{Tensor absolute value equations}

 In this section, we present some basic definitions and properties in absolute value equations, nonlinear
 complementarity problems, and nonsmooth analysis, which will be used in the sequel.

We recall the absolute value equations (AVE) of the type
\begin{equation}\label{ave}
A{\bf x}-|{\bf x}|={\bf b},
\end{equation}
where $A\in R^{n\times n}$, ${\bf b}\in R^n$ and $|{\bf x}|$ denotes
the vector with absolute values of each component of ${\bf x}$. The
AVE (\ref{ave}) has been widely investigated in many literatures
such as \cite{mangas1,mangas2,mangas3}. In \cite{mangas3}, some
results about the AVE are given, which we list as follows:
\begin{itemize}
\item[(i)] The AVE (\ref{ave}) is equivalent to the bilinear program
$$
0=\min\{((A+I){\bf x}-{\bf b})^T((A-I){\bf x}-{\bf b})|\, (A+I){\bf
x}-{\bf b}\ge {\bf 0}, (A-I){\bf x}-{\bf b}\ge {\bf 0}\},
$$
and the generalized linear complementarity problem
$$
(A+I){\bf x}-{\bf b}\ge {\bf 0},\quad  (A-I){\bf x}-{\bf b}\ge {\bf
0} \quad ((A+I){\bf x}-{\bf b})^T((A-I){\bf x}-{\bf b})=0.
$$
\item[(ii)] Let $C\in R^{n\times n}$ and ${\bf b}\in R^n$. Then
$$
(C-I){\bf z}={\bf b}, \quad {\bf z}\ge {\bf 0}\quad  \mbox{has a
solution
 ${\bf z}\in R^n$}
 $$
  implies that
  $$
  A{\bf x}-|{\bf x}|={\bf b} \quad \mbox{has a solution for any $A=CD$ with $D=diag(\pm 1)$}.$$
\end{itemize}

Clearly, the tensor absolute value equation (\ref{TAVE}) is a
generalization of the AVE (\ref{ave}) from the matrix case to the
tensor case.  Take an equation with the coefficient tensor
$\mathcal{A}\in R^{2\times 2\times 2}$ as an example. The tensor
absolute equation
$$\mathcal{A}{\bf x}^2-|{\bf x}|^2={\bf b}$$ is a condense form of
$$
\left\{\begin{array}{l}a_{111}x_1^2 + (a_{112} + a_{121})x_1x_2 +
a_{122}x_2^2-|x_1|^2 = b_1,\\
 a_{211}x^2_1 + (a_{212} + a_{221})x_1x_2 + a_{222}x^2_2-|x_2|^2 = b_2.
 \end{array}\right.
 $$
 We want to find $x_1$ and $x_2$ that satisfy the above two
equations.

The following example shows a specific tensor absolute value
equation.
\begin{example}\label{eg1} Let a tensor $\mathcal{A}\in T(4,2)$ be defined by
$a_{1111}=a_{2111}=a_{2222}=1$, $a_{1222}=-1$, and zero otherwise.
Let ${\bf b}=(1,2)^T$. Then the corresponding tensor absolute value
equation is \begin{equation}\label{equ1} \left\{
\begin{array}{rcl}
  x_1^3-x_2^3- {|x_1|}^3&=&1,\\
-2x_1^3+x_2^3- {|x_2|}^3&=&2.
  \end{array}
\right.\end{equation}
\end{example}
By simplicity computation, we see that the TAVE (\ref{equ1}) in
Example \ref{eg1} has no solution. In the next section we will
discuss the existence of solution for the TAVE (\ref{TAVE}). We can
extend the result (ii) to the TAVE and obtain a similar condition
for the existence of solution to  (\ref{TAVE}).

 Below, we introduce the classical nonlinear
complementarity problem. The tensor complementarity problem recently
introduced in \cite{songqi} is a special kind of nonlinear
complementarity problem. It will be shown in Section 4 that the TAVE
(\ref{TAVE}) can be reformulated as a special kind of generalized
tensor complementarity problem.

\begin{definition} Given a given mapping $F: R^n\to R^n$, the nonlinear
complementarity problem, denoted by NCP($F$), is to find a vector
${\bf x}\in R^n$ satisfying
$$
{\bf x}\ge {\bf 0},\quad F({\bf x})\ge {\bf 0},\quad {\bf x}^TF({\bf
x})=0.
$$
\end{definition}

Many solution methods developed for NCP($F$) or related problems are
based on reformulating them as a system of equations using so-called
NCP-functions \cite{NCP}. Here a function $\phi: R^2\to R$ is called
an NCP-function if
$$
\phi(a,b)=0\quad \Leftrightarrow \quad a\ge 0, b\ge 0, ab=0.
$$
Given an NCP-function $\phi$, let us define
$$
\Phi({\bf x})=\left(\phi(x_1,F_1({\bf x})),\ldots,\phi(x_n,F_n({\bf
x}))\right)^T.
$$
It is obvious that ${\bf x}\in R^n$ is a solution of NCP($F$) if and
only if it solves the system of nonsmooth equations $$\Phi({\bf
x})={\bf 0}.$$ For the solution of $\Phi({\bf x})={\bf 0}$, we
recall some definitions in nonsmooth analysis. Suppose that $\Theta:
U\subseteq R^{n_1}\to R^{n_2}$ is a locally Lipschitz function,
where $U$ is nonempty and open. By Rademacher's Theorem, $\Theta$ is
differentiable almost everywhere. Let $D_{\Theta}\subseteq R^{n_1}$
denote the set of points at which $\Theta$ is differentiable. For
any ${\bf x}\in D_{\Theta}$, we write $J\Theta({\bf x})$ for the
usual $n_2\times n_1$ Jacobian matrix of partial derivatives. The
$B$-{\it subdifferential} of $\Theta$ at ${\bf x}\in U$ is the set
defined by
$$
\partial_B \Theta({\bf x})=\left\{V\in R^{n_2\times n_1}|\, \mbox{$\exists \{{\bf x}^k\}\subseteq D_{\Theta}$ with ${\bf x}^k\to {\bf x}$, $J\Theta({\bf x}^k)\to V$}\right\}.
$$
The Clarke's {\it generalized Jacobian} of $\Theta$ at ${\bf x}$ is
the set defined by
$$
\partial \Theta({\bf x})=\mbox{co($
\partial_B \Theta({\bf x})$)},
$$
where ``co" denotes the convex hull. Then, $
\partial \Theta({\bf x})$ is a nonempty convex compact subset of $R^{n_2\times n_1}$ \cite{clark9}.
The function $\Theta$ is {\it semismooth} \cite{hin17,qi32} at ${\bf
x}\in R^{n_1}$ if
$$
\lim_{\begin{subarray}{c}V\in \partial \Theta({\bf x}+t\tilde{{\bf d}})\\
\tilde{{\bf d}}\to {\bf d},\,t\downarrow
0\end{subarray}}V\tilde{d}$$ exists for all ${\bf d}\in R^{n_1}$. If
$\Theta$ is semismooth at all ${\bf x}\in U$, we call $\Theta$
semismooth on $U$. The function $\Theta$ is called {\it strongly
semismooth} \cite{qi33} if it is semismooth and for any ${\bf x}\in
U$ and $V\in \partial \Theta({\bf x}+t{\bf d})$,
$$
Vd-\Theta'({\bf x};{\bf d})=O(\|{\bf d}\|^2),\quad {\bf d}\to {\bf
0},$$ where $\Theta'({\bf x};{\bf d})$ denotes the directional
derivative \cite{bon3} of $\Theta$ at ${\bf x}$ in direction ${\bf
d}$, i.e.,
$$
\Theta'({\bf x};{\bf d})=\lim_{t\downarrow 0}\frac{\Theta({\bf
x}+t{\bf d})-\Theta({\bf x})}{t}.
$$
Note that if the function $\Theta$ is semismooth at ${\bf x}$, the
directional derivative $\Theta'({\bf x};{\bf d})$ exists for all
${\bf d}\in R^{n_1}$ and
$$
\Theta'({\bf x};{\bf d})=\lim_{\substack{ V\in \partial \Theta({\bf x}+t\tilde{{\bf d}})\\
\tilde{{\bf d}}\to {\bf d},\, t\downarrow 0}}V\tilde{d}.$$

We now present some NCP-functions which are widely used in nonlinear
complementarity problems. For more details about NCP-functions and
their smoothing approximations, one can refer to
\cite{qisunzhou,zhou} and references therein.

Here we give some well-known NCP-functions as follows:
\begin{itemize}
\item The min function: $$\phi(a,b)=\min\{a,b\}.$$

\item The Fischer-Burmeister function:
$$\phi_{FB}(a,b)=a+b-\sqrt{a^2+b^2}.$$
\end{itemize}
It has been shown that all these NCP-functions are globally
Lipschitz continuous, directionally differentiable, and strongly
semismooth \cite{fbf,minf}. For example, the generalized gradient
$\partial\phi_{FB}(a,b)$ of $\phi_{FB}(a,b)$ is equal to the set of
all $(v_a,v_b)$ such that
$$
(v_a,v_b)=\left\{\begin{array}{ll}\left(1-\frac{a}{\sqrt{a^2+b^2}}, 1-\frac{b}{\sqrt{a^2+b^2}}\right) & \quad \mbox{if $(a,b)\ne (0,0)$},\\
(1-\xi,1-\varsigma) & \quad \mbox{if $(a,b)= (0,0)$},
\end{array}\right.
$$
where $(\xi,\varsigma)$ is any vector satisfying
$\xi^2+\varsigma^2\le 1$.

In Section 4, we will use the Fischer-Burmeister function to
reformulate the TAVE (\ref{TAVE}) as  a system of equations and then
we will propose an algorithm to solve the system of equations.

We now introduce the tensor complementarity problem which first
defined by Song and Qi \cite{songqi}.
\begin{definition} Given any given tensor $\mathcal{A}\in T(m,n)$  and vector ${\bf q}\in R^n$, the
tensor complementarity problem, denoted by TCP($\mathcal{A},{\bf
q}$), is to find a vector ${\bf x}\in R^n$ satisfying
$$
{\bf x}\ge {\bf 0},\quad \mathcal{A}{\bf x}^{m-1}+{\bf q}\ge {\bf
0},\quad {\bf x}^T(\mathcal{A}{\bf x}^{m-1}+{\bf q})=0.
$$
\end{definition}
Note that when $n = 2$, the tensor $\mathcal{A}$ reduces to a
matrix, denoted by $A$, and the TCP($\mathcal{A},{\bf q}$) becomes:
find  a vector ${\bf x}\in R^n$ such that
$$
{\bf x}\ge {\bf 0},\quad A{\bf x}+{\bf q}\ge {\bf 0},\quad {\bf
x}^T(A{\bf x}+{\bf q})=0,
$$
 which is just the
linear complementarity problem \cite{LCP}. Very recently, a class of
$n$-person noncooperative games are in \cite{huangqi}, where the
utility function of every player is given by a homogeneous
polynomial defined by the payoff tensor of that player, which is a
natural extension of the bimatrix game where the utility function of
every player is given by a quadratic form defined by the payoff
matrix of that player. Such a problem is called the multilinear
game. The multilinear game is reformulated as a tensor
complementarity problem. Some semismooth Newton-type methods are
recently proposed for solving the tensor complementarity problems
(see, e.g., \cite{chenqi}). In Section 4, we will extend the result
(i) to the TAVE  (\ref{TAVE}) and show that the TAVE  (\ref{TAVE})
is equivalent to a bi-multilinear program and a generalized tensor
complementarity problem.

\section{Existence of solutions}

In this section, we give some sufficient conditions for the
existence of solutions to the TAVE (\ref{TAVE}). Specially, we
extend the result (ii) about the AVE (\ref{ave}) to the TAVE
(\ref{TAVE}).

We need the following lemmas which are recently established in
\cite[Theorems 3.2, 3.3, 3.4]{dingwei}.
\begin{lemma}\label{lem1}
Let $\mathcal{A}\in T(m,n)$. If $\mathcal{A}$ is a strong
$M$-tensor, then for every positive vector ${\bf b}$ the multilinear
system of equations $\mathcal{A}{\bf x}^{m-1}={\bf b}$ has a unique
positive solution.
\end{lemma}

\begin{lemma}\label{lem2}
Let $\mathcal{A}\in T(m,n)$ be a $Z$-tensor. Then it  is a strong
$M$-tensor if and only if   the multilinear system of equations
$\mathcal{A}{\bf x}^{m-1}={\bf b}$ has a unique positive solution
for every positive vector ${\bf b}$.
\end{lemma}

\begin{lemma}\label{lem3}
Let $\mathcal{A}\in T(m,n)$ be an $M$-tensor and ${\bf b}\ge {\bf
0}$.  If  there exists ${\bf v}\ge {\bf 0}$ such that
$\mathcal{A}{\bf v}^{m-1}\ge {\bf b}$, then the multilinear system
of equations $\mathcal{A}{\bf x}^{m-1}={\bf b}$ has a nonnegative
solution.
\end{lemma}

By the above lemmas, we have the following theorems.
\begin{theorem}\label{thm1}
Let $\mathcal{A}\in T(m,n)$. If $\mathcal{A}$ can be written as
$\mathcal{A}=c\mathcal{I}-\mathcal{B}$ with $\mathcal{B}\ge 0$ and
$c>\rho(\mathcal{B})+1$, then for every positive vector ${\bf b}$
the TAVE (\ref{TAVE}) has a unique positive solution.
\end{theorem}
\begin{proof}
Let $s=c-1$. Then $\mathcal{A}=c\mathcal{I}-\mathcal{B}$ yields
$$\mathcal{A}-\mathcal{I}=s\mathcal{I}-\mathcal{B},\quad \mathcal{B}\ge 0,\quad s>\rho(\mathcal{B}),$$
which implies that $\mathcal{A}-\mathcal{I}$ is a strong $M$-tensor.
By Lemma \ref{lem1}, the multilinear system of equations
$$(\mathcal{A}-\mathcal{I}){\bf x}^{m-1}={\bf b}$$ has a unique positive solution for every positive vector ${\bf b}$. Hence,
for every positive vector ${\bf b}$, the TAVE (\ref{TAVE}) has a
unique positive solution. %\qed
\end{proof}

Combining \cite[Theorem 3]{dingqiwei} and Lemma \ref{lem2}, we can
rewrite the above theorem into an equivalent condition for
$\mathcal{A}-\mathcal{I}$ being a strong $M$-tensor.
\begin{theorem}\label{thm2}
Let $\mathcal{A}\in T(m,n)$ be a $Z$-tensor. Then $\mathcal{A}$ can
be written as  the form of
\begin{equation}\label{equ2}\mathcal{A}=c\mathcal{I}-\mathcal{B}, \quad\mathcal{B}\ge 0,\quad c>\rho(\mathcal{B})+1 \end{equation}
 if and only if for every positive vector ${\bf b}$ the TAVE (\ref{TAVE}) has a unique positive solution.
\end{theorem}
\begin{proof}
On one hand, by Theorem \ref{thm1}, we have the existence and
uniqueness of the positive solution of the TAVE (\ref{TAVE}) for
every positive vector ${\bf b}$. On the other hand, if for every
positive vector ${\bf b}$ the TAVE (\ref{TAVE}) has a unique
positive solution, then there exists a vector ${\bf x}> {\bf 0}$
such that $$(\mathcal{A}-\mathcal{I}){\bf x}^{m-1}={\bf b}>{\bf
0}.$$ Since $\mathcal{A}$  is a $Z$-tensor,
$\mathcal{A}-\mathcal{I}$  is also a $Z$-tensor. Thus, by
\cite[Theorem 3]{dingqiwei}, $\mathcal{A}-\mathcal{I}$ is a strong
$M$-tensor and then the form of (\ref{equ2}) holds. %\qed
\end{proof}

{\sc Remark.} The sufficient condition in Theorem \ref{thm2} can be
weakened as follows: if the TAVE (\ref{TAVE}) has a nonnegative
solution for every positive vector ${\bf b}$, then we also have the
form (\ref{equ2}). In fact, let ${\bf x}\ge {\bf 0}$ be a solution
of the TAVE (\ref{TAVE}). Then there exists ${\bf x}\ge {\bf 0}$
such that $(\mathcal{A}-\mathcal{I}){\bf x}^{m-1}>{\bf 0}$. By
\cite[Theorem 3]{dingqiwei}, we can obtain the conclusion.

\begin{theorem}
Let ${\bf b}\ge {\bf 0}$ and  $\mathcal{A}\in T(m,n)$ be in the form
of  $\mathcal{A}=c\mathcal{I}-\mathcal{B}$ with $\mathcal{B}\ge 0$
and $c=\rho(\mathcal{B})+1$. If there exists a vector ${\bf v}\ge
{\bf 0}$ such that $(\mathcal{A}-\mathcal{I}){\bf v}^{m-1}\ge {\bf
b}$, then  the TAVE (\ref{TAVE}) has a nonnegative solution.
\end{theorem}
\begin{proof}
It follows from
 $$\mathcal{A}=c\mathcal{I}-\mathcal{B},\quad \mathcal{B}\ge 0, \quad c=\rho(\mathcal{B})+1 $$
 that $\mathcal{A}-\mathcal{I}$ is an $M$-tensor. By Lemma \ref{lem3}, there is ${\bf x}^*\ge {\bf 0}$ such that
 $$
 (\mathcal{A}-\mathcal{I})({\bf x}^*)^{m-1}={\bf b}.$$
 Thus, we have
 $$
 \mathcal{A}({\bf x}^*)^{m-1}-|{\bf x}^*|^{m-1}={\bf b}.$$
This completes the proof. %\qed
\end{proof}

We next extend the result (i) about the AVE (\ref{ave}) to the TAVE
(\ref{TAVE}). Here, we assume that $m$ is even. We first introduce
the product of a tensor and a diagonal tensor.
\begin{definition}\label{def2}
Let $\mathcal{C}=(c_{i_1i_2\ldots i_m})\in T(m,n)$ and
$\mathcal{B}\in T(m,n)$ be a diagonal tensor with diagonal elements
$b_{i\ldots i}$. We denote $\mathcal{CB}=(a_{i_1i_2\ldots i_m})$
their product, whose elements are defined as
$$
a_{i_1i_2\ldots i_m}=c_{i_1i_2\ldots i_m}(b_{i_2\ldots
i_2})^{\frac{1}{m-1}}\cdots (b_{i_m\ldots i_m})^{\frac{1}{m-1}},
\quad 1\le i_1,\cdots, i_m\le n.$$
\end{definition}
Obviously, Definition \ref{def2} is well-defined due to the
assumption that $m$ is even.

By simplicity computation, we have the following proposition.
\begin{proposition}\label{prop1}
Let $\mathcal{C}=(c_{i_1i_2\ldots i_m})\in T(m,n)$ and ${\bf x}\in
R^n$. We have
$$(\mathcal{CD}){\bf x}^{m-1}= \mathcal{C}\left(\mathcal{D}{\bf x}^{m-1}\right).$$
\end{proposition}
\begin{proof} Let us define a vector ${\bf u}\in R^n$ as
$$
{\bf u}=\left(\mathcal{D}{\bf x}^{m-1}\right)^{[\frac1{m-1}]}.$$
Then by some definitions introduced in Section 1, the
$i$th-component of the vector $\mathcal{C}(\mathcal{D}{\bf
x}^{m-1})$  can be written as
%\begin{equation}\label{equpro1}
%\begin{array}{rcl}
\begin{eqnarray}
\left(\mathcal{C}(\mathcal{D}{\bf x}^{m-1})\right)_i&=&(\mathcal{C}{\bf u}^{m-1})_i\nonumber \\
&=&\sum\limits_{i_2=1}^n\cdots \sum\limits_{i_m=1}^nc_{ii_2\ldots i_m}u_{i_2}\cdots u_{i_m}\label{equpro1}\\
&=& \sum\limits_{i_2=1}^n\cdots \sum\limits_{i_m=1}^nc_{ii_2\ldots i_m}\left(d_{i_2\ldots i_2}x_{i_2}^{m-1}\right)^{\frac1{m-1}}\cdots \left(d_{i_m\ldots i_m}x_{i_m}^{m-1}\right)^{\frac1{m-1}}\nonumber\\
&=&\sum\limits_{i_2=1}^n\cdots \sum\limits_{i_m=1}^nc_{ii_2\ldots
i_m}d_{i_2\ldots i_2}^{\frac1{m-1}}\cdots d_{i_m\ldots
i_m}^{\frac1{m-1}}x_{i_2}\cdots x_{i_m}.\nonumber
\end{eqnarray}
%\end{array}
%\end{equation}
Let $\mathcal{A}=\mathcal{C}\mathcal{D}$. Then by Definition
\ref{def2}, the $i$th-component of the vector
$(\mathcal{C}\mathcal{D}){\bf x}^{m-1}$ can be written as
\begin{equation}\label{equpro2}
\begin{array}{rcl}
\left((\mathcal{C}\mathcal{D}){\bf x}^{m-1}\right)_i&=&\sum\limits_{i_2=1}^n\cdots \sum\limits_{i_m=1}^na_{ii_2\ldots i_m}x_{i_2}\cdots x_{i_m}\\
&=&\sum\limits_{i_2=1}^n\cdots \sum\limits_{i_m=1}^nc_{ii_2\ldots
i_m}d_{i_2\ldots i_2}^{\frac1{m-1}}\cdots d_{i_m\ldots
i_m}^{\frac1{m-1}}x_{i_2}\cdots x_{i_m}.
\end{array}
\end{equation}
Combining (\ref{equpro1}) and (\ref{equpro2}), we have
$$
(\mathcal{C}\mathcal{D}){\bf x}^{m-1}=\mathcal{C}(\mathcal{D}{\bf
x}^{m-1}).
$$
%\qed
\end{proof}

It is easy to see that
\begin{equation}\label{equ3}
|{\bf x}|^{m-1}=\mathcal {D}{\bf x}^{m-1}
\end{equation}
holds for any vector ${\bf x}\in R^n$, because the $i$th-component
of the vectors $|{\bf x}|^{m-1}$ and $\mathcal {D}{\bf x}^{m-1}$ are
in the form of
$$
\left(|{\bf x}|^{m-1}\right)_i=|x_i|^{m-1},\quad \left(\mathcal
{D}{\bf x}^{m-1}\right)_i=d_{i\ldots i}x_i^{m-1}.
$$
Here the sign of $x_i$ is corresponded to the diagonal element $1$
or $-1$ of $\mathcal{D}$.

The following theorem is a generalization of the result (ii) from
AVE to TAVE.
\begin{theorem}\label{thm4}
Let $\mathcal{C}\in T(m,n)$, ${\bf b}\in R^n$ and
$\mathcal{A}=\mathcal{C}\mathcal{D}$. If the multilinear system of
equations
\begin{equation}\label{multi1}
(\mathcal{C}-\mathcal{I}){\bf z}^{m-1}={\bf b}, \quad {\bf z}\ge
{\bf 0}
\end{equation}
has a solution, then the tensor absolute value equation
$$
\mathcal{A}{\bf x}^{m-1}-|{\bf x}|^{m-1}={\bf b}
$$
also has a solution.
\end{theorem}
\begin{proof} Let ${\bf z}^*$ be the solution of the multilinear system of equations (\ref{multi1}). Then we have
\begin{equation}\label{equthm}
(\mathcal{C}-\mathcal{I})({\bf z}^*)^{m-1}={\bf b},\quad {\bf
z}^*\ge {\bf 0}.
\end{equation}
Take
$$
\mathcal{D}({\bf x}^*)^{m-1}=({\bf z}^*)^{m-1}.$$  Then
(\ref{equthm}) can be rewritten as
$$
\mathcal{C}(\mathcal{D}({\bf x}^*)^{m-1})-\mathcal{D}({\bf
x}^*)^{m-1}={\bf b},
$$
which, together with Proposition \ref{prop1} and (\ref{equ3}),
implies that ${\bf x}^*$ is a solution of the tensor absolute value
equation
$$
\mathcal{A}{\bf x}^{m-1}-|{\bf x}|^{m-1}={\bf b}.
$$ Thus, we complete the proof. %\qed
\end{proof}

We give an example to verify the above theorem.
\begin{example}
Let $\mathcal{C}\in T(4,2)$ with
$c_{1111}=c_{1222}=c_{2111}=c_{2222}=1$ and zero otherwise, and
${\bf b}=(8,8)^T$. Consider the multilinear system of equations
$$
(\mathcal{C}-\mathcal{I}){\bf z}^{m-1}={\bf b}.
$$
It is rewritten as
$$
z_1^3=8,\quad z_2^3=8.$$ This implies that ${\bf z}^*=(2,2)^T$ is a
solution of
$$
(\mathcal{C}-\mathcal{I}){\bf z}^{m-1}={\bf b},\quad {\bf z}\ge {\bf
0}.
$$
Let $\mathcal{D}\in T(4,2)$ be a diagonal tensor with $d_{1111}=1$
and $d_{2222}=-1$. Then we have $\mathcal{A}\in T(4,2)$ with
$a_{1111}=a_{2111}=1$, $a_{1222}=a_{2222}=-1$, and zero otherwise,
i.e., $\mathcal{A}=\mathcal{C}\mathcal{D}$. By Theorem \ref{thm4},
${\bf x}^*=(2,-2)^T$ is just a solution of the tensor absolute value
equation
\begin{equation}\label{equeg}
\mathcal{A}{\bf x}^{m-1}-|{\bf x}|^{m-1}={\bf b}.
\end{equation}
We now verify the conclusion. We rewrite (\ref{equeg}) as
$$
\left\{
\begin{array}{rcl}
x_1^3-x_2^3-|x_1|^3&=&8,\\
-x_2^3+x_1^3-|x_2|^3&=&8.\\
\end{array}\right.
$$
By simplicity computation,  the above equation has a solution
$x_1=2, x_2=-2$.
\end{example}

\section{Reformulation and algorithm}
In this section, we extend the result (i) from AVE to TAVE. We show
that the TAVE (\ref{TAVE}) is equivalent to a bi-multiliear program
and a generalized tensor complementarity problem. We first introduce
the following definition.
\begin{definition}
Let $\mathcal{A}\in T(m,n)$ and ${\bf x}, {\bf b}\in R^n$. Define
$$F({\bf x})=(\mathcal{A}+\mathcal{I}){\bf x}^{m-1}-{\bf b}, \quad G({\bf x})=(\mathcal{A}-\mathcal{I}){\bf x}^{m-1}-{\bf b}.$$
The generalized tensor complementarity problem is to find ${\bf
x}\in R^n$ satisfying
\begin{equation}\label{gtave}
F({\bf x})\ge {\bf 0},\quad G({\bf x})\ge {\bf 0},\quad F({\bf
x})^TG({\bf x})=0.
\end{equation}
We call the following nonlinear program as a bi-multiliear program:
\begin{equation}\label{biprog}
 0=\min\left\{F({\bf x})^TG({\bf x})|\, F({\bf x})\ge {\bf 0}, G({\bf x})\ge {\bf 0}\right\}
\end{equation}
\end{definition}

\begin{theorem}Let $\mathcal{A}\in T(m,n)$ and ${\bf b}\in R^n$. Then the TAVE (\ref{TAVE}) is equivalent to the generalized tensor complementarity problem (\ref{gtave}) and the bi-multilinear program (\ref{biprog}).
\end{theorem}
\begin{proof}
Clearly, the generalized tensor complementarity problem
(\ref{gtave}) is equivalent to the bi-multilinear program
(\ref{biprog}). That is, $(\ref{gtave})\Leftrightarrow
(\ref{biprog})$.

We only need to prove $(\ref{TAVE}) \Leftrightarrow (\ref{biprog})$.
In fact, $|{\bf x}|^{m-1}=|{\bf x}^{m-1}|$. Hence, we have
$$
|{\bf x}|^{m-1}\le \mathcal{A}{\bf x}^{m-1}-{\bf b}  \Leftrightarrow
(\mathcal{A}+\mathcal{I}){\bf x}^{m-1}-{\bf b}\ge {\bf 0},\quad
(\mathcal{A}-\mathcal{I}){\bf x}^{m-1}-{\bf b}\ge {\bf 0}.
$$
This implies that ${\bf x}$ is a feasible solution of
(\ref{biprog}). Since
$$|{\bf x}|^{m-1}= \mathcal{A}{\bf x}^{m-1}-{\bf b}  \Leftrightarrow \left((\mathcal{A}+\mathcal{I}){\bf x}^{m-1}-{\bf b}\right)^T\left((\mathcal{A}-\mathcal{I}){\bf x}^{m-1}-{\bf b}\right)=0,
$$
we have
$$ |{\bf x}|^{m-1}= \mathcal{A}{\bf x}^{m-1}-{\bf b}  \Leftrightarrow  0=\min\left\{F({\bf x})^TG({\bf x})|\, F({\bf x})\ge {\bf 0}, G({\bf x})\ge {\bf 0}\right\}.
$$
This completes the proof. %\qed
\end{proof}

By the above theorem, in order to solve the TAVE (\ref{TAVE}), we
propose an algorithm for solving the generalized tensor
complementarity problem (\ref{gtave}). Using the Fischer-Burmeister
function $\phi_{FB}$, we can reformulate (\ref{gtave}) as the
following equation:
$$
H({\bf x})=\left(\begin{array}{c}
\phi_{FB}(F_1({\bf x}), G_1({\bf x}))\\
\vdots \\
\phi_{FB}(F_n({\bf x}), G_n({\bf x}))
\end{array}\right)={\bf 0}.
$$
Hence, ${\bf x}$ is a solution of (\ref{TAVE}) if and only if
$H({\bf x})={\bf 0}$. Moreover, $H({\bf x})$ is strongly semismooth
since the composition of strongly semismooth function is again
strongly semismooth \cite{miff26}, and according to the Jacobian
chain rule, we have the following result.
\begin{theorem}\label{thm6} Let $\mathcal{A}\in S(m,n)$. Then the function $H({\bf x})$ is strongly semismooth. Moreover, for any ${\bf x}\in R^n$, we have
$$\partial H({\bf x})\subseteq D_a({\bf x})JF({\bf x})+D_b({\bf x})JG({\bf x}),$$
where $D_a({\bf x})=diag(a_i({\bf x}))$ and  $D_b({\bf
x})=diag(b_i({\bf x}))$ are diagonal matrices in $R^{n\times n}$
with entries $(a_i({\bf x}),b_i({\bf x}))\in \partial
\phi_{FB}(F_i({\bf x}), G_i({\bf x}))$, where $\partial
\phi_{FB}(F_i({\bf x}), G_i({\bf x}))$ denotes the set $\partial
\phi_{FB}(a,b)$ with $(a,b)$ being replaced by $(F_i({\bf x}),
G_i({\bf x}))$, and  $JF({\bf x})$ and $JG({\bf x})$ are given by
$$
JF({\bf x})=(\mathcal{A}+\mathcal{I}){\bf x}^{m-2}, \quad JG({\bf
x})=(\mathcal{A}-\mathcal{I}){\bf x}^{m-2}.
$$
Here, for a tensor $\mathcal{T}=(t_{i_1\ldots i_m})\in T(m,n)$ and a
vector ${\bf x}\in R^n$, let $\mathcal{T}{\bf x}^{m-2}$ be a matrix
in $R^{n\times n}$ whose $(i,j)$-th component is defined by
$$
\left(\mathcal{T}{\bf x}^{m-2}\right)_{ij}=\sum_{i_3}^n\cdots
\sum_{i_m}^n t_{iji_3\ldots i_m}x_{i_3}\cdots x_{i_m}.
$$
\end{theorem}

In order to propose an algorithm for the solution of $H({\bf
x})={\bf 0}$, we define a merit function as
$$
\Psi({\bf x})=\frac12\|H({\bf x})\|^2.
$$
We present some properties of the merit function, which can be
obtained by \cite[Theorem 2.2.4 and Theorem 2.6.6]{clark9}.
\begin{theorem}
Let $\mathcal{A}\in S(m,n)$. Then the merit function $\Psi({\bf x})$
is continuously differentiable with $$\nabla \Psi({\bf x})=Q^TH({\bf
x})$$ for any $Q\in \partial H({\bf x})$.
\end{theorem}

We are now in the position to propose a Levenberg-Marquardt-type
algorithm to solve the semismooth system of equations $H({\bf
x})={\bf 0}$, which is an extension of the nonsmooth inexact
Levenberg-Marquardt-type method in \cite{FK}. To ensure global
convergence, a line search is performed to minimize the smooth merit
function $\Psi$. Because the problem with data in a structure of
tensor is large scale, and the inexact version is more suited to the
large-scale case \cite{FK}, we have the following algorithm.

\begin{algorithm}\label{alg4.1}
{\rm (Inexact Levenberg-Marquardt-type method)}

\alglist
\item[{\bf Step 0.}] Given a starting vector ${\bf x}^0\in R^{n}$ and some scales $p>2$, $0<\beta<1/2$,  $\rho>0$, $ \epsilon\ge 0$. Set $k:=0$.

\item[{\bf Step 1.}]  If $\|H({\bf x}^k)\|\leq\epsilon$, stop. Otherwise,  compute $Q^k\in \partial H({\bf x}^k)$.

\item[{\bf Step 2.}] Find a solution ${\bf d}^k$ satisfying
\begin{equation}\label{step2}\left((Q^k)^TQ^k+\mu_kI\right){\bf d}=-(Q^k)^TH({\bf x}^k)+{\bf
r}^k,\end{equation} where $\mu_k\ge 0$ is the Levenberg-Marquardt
parameter. If the condition
$$\nabla\Psi({\bf x}^k)^T{\bf d}^k\le -\rho\|{\bf d}^k\|^p$$ is not satisfied, set $${\bf d}^k=-\nabla\Psi({\bf x}^k).$$

\item[{\bf Step 3.}] Find the smallest integer $i^k\in \{0,1,2,\ldots \}$ such
that $t_k=2^{-i^k}$ and
  $$\Psi({\bf x}^k+ t_k {\bf d}^k)\leq\Psi({\bf x}^k)+\beta t_k\nabla\Psi({\bf x}^k)^T{\bf d}^k. $$

\item[{\bf Step 4.}] Set ${\bf x}^{k+1}={\bf x}^k+t_k {\bf d}^k$, $k:=k+1$, and go to Step 1.
\end{list}
\end{algorithm}

In what follows, we analyze  the global convergence of Algorithm 1.
We shall assume that Algorithm 1 produce an infinite sequence
$\{{\bf x}^k\}$. By \cite[Theorem 15 and Theorem 16]{FK}, we
immediately obtain the following theorems.
\begin{theorem}\label{thmf}
 Assume that the sequence $\{\mu_k\}$ is bounded and that the sequence $\{{\bf r}^k\}$ satisfies
 $$\|{\bf r}^k\|\leq \alpha_k\|\nabla\Psi({\bf x}^k)\|,$$ where $\{\alpha_k\}$ is a sequence of  numbers with
 $0<\alpha_k <1$  and $\alpha_k\to 0$ as $k\to \infty$. Then each accumulation point of  $\{{\bf x}^k\}$  is a stationary point of $\Psi$.
\end{theorem}

\begin{theorem}Let the assumptions of Theorem \ref{thmf} hold. If one of the accumulation points of
$\{{\bf x}^k\}$, denoted ${\bf x}^*$, is an isolated solution of the
TAVE (\ref{TAVE}), then $$\lim_{k\to \infty} {\bf x}^k={\bf x}^*.$$
\end{theorem}

In the implementation of Algorithm \ref{alg4.1}, the computational
most intensive part is the approximation solution of system
(\ref{step2}) with ${\bf r}^k={\bf 0} ~~~\forall k$. We note that
the system is always solvable. In fact, if $\mu_k>0$, the matrix
$(Q^k)^TQ^k+\mu_kI$ is symmetric positive definite and hence system
(\ref{step2}) is surely solvable. If $\mu_k=0$, the matrix
$(Q^k)^TQ^k+\mu_kI$ reduces to $(Q^k)^TQ^k$, which is guaranteed to
be only positive semidefinite. However, in this case, (\ref{step2})
reduces to the normal gradient equation $Q^k{\bf d}=-H({\bf x}^k)$,
is therefore solvable. We now have to specify which element $Q^k\in
\partial H({\bf x}^k)$ we select at the $k$-th iteration. By Theorem
\ref{thm6}, we have that an element of $\partial H({\bf x}^k)$ can
be obtained in the following way. Let
$$
\Lambda =\{i:\, F_i({\bf x}^k)=0=G_i({\bf x}^k)\}
$$
be the set of ``degenerate indices'' and define ${\bf z}\in R^n$ to
be a vector whose components $z_i$ are $1$ if $i\in \Lambda$ and $0$
otherwise. Then, the matrix $Q^k$ defined by
$$
Q^k=A({\bf x}^k)JF({\bf x}^k)+B({\bf x}^k)JG({\bf x}^k),
$$
where $A$ and $B$ are $n\times n$ diagonal matrices whose $i$-th
diagonal elements are given, respectively, by
$$
A_{ii}({\bf x}^k)=\begin{cases}1-\frac{F_i({\bf
x}^k)}{\sqrt{F_i^2({\bf x}^k)+G_i^2({\bf x}^k)}} & \mbox{if
$i\not\in \Lambda$},\\
1-\frac{\nabla F_i({\bf x}^k)^Tz}{\sqrt{(\nabla F_i({\bf
x}^k)^Tz)^2+(\nabla G_i({\bf x}^k)^Tz)^2}} & \mbox{if $i\in
\Lambda$},
\end{cases}
$$
and by
$$
B_{ii}({\bf x}^k)=\begin{cases}1-\frac{G_i({\bf
x}^k)}{\sqrt{F_i^2({\bf x}^k)+G_i^2({\bf x}^k)}} & \mbox{if
$i\not\in \Lambda$},\\
1-\frac{\nabla G_i({\bf x}^k)^Tz}{\sqrt{(\nabla F_i({\bf
x}^k)^Tz)^2+(\nabla G_i({\bf x}^k)^Tz)^2}} & \mbox{if $i\in
\Lambda$},
\end{cases}
$$
belongs to $\partial H({\bf x}^k)$. In the next section, we compute
$Q^k$ as the formulation.

\section{Numerical results}

In this section, we present the numerical performance of Algorithm
\ref{alg4.1} for the TAVE (\ref{TAVE}). All codes were written by
using Matlab Version R2015b and Tensor Toolbox Version 2.6
\cite{bader2}. The numerical experiments were done on a laptop with
an Intel Core i7-4720HQ CPU (2.6GHz) and RAM of 7.89GB.

In the implementation of Algorithm \ref{alg4.1}, we set
$\varepsilon=10^{-6},\rho=10^{-10},p=2.1,\beta=10^{-4}$ and the
Levenberg-Marquardt parameter $\mu_k=0.3\,\,\forall k\in N$. We also
set a maximum iteration steps for the algorithm, $i.e.$,
$N_{max}=300$.

The first numerical experiment focuses on the behaviour of
algorithm's iteration. We generate a random symmetric nonnegative
tensor $\mathcal{A}\in S_{6,8}$ and a random vector ${\bf x}^*\in
R^8$. All entries of $\mathcal{A}$ and ${\bf x}^*$ are uniform
random numbers in the interval $[0,1]$. We calculate ${\bf
b}=\mathcal{A}{{\bf x}^*}^{m-1}-|{\bf x}^*|^{m-1}$ in order to make
TAVE have at least one solution. Then we use Algorithm \ref{alg4.1}
to solve TAVE: $\mathcal{A}{\bf x}^{m-1}-|{\bf x}|^{m-1}={\bf b}$,
with a random initial point chosen randomly from $[0,1]^8$ which is
shown as $x^0$ in table \ref{tab:1}. The iteration of Algorithm
\ref{alg4.1} is shown in Table \ref{tab:1}. From the table,
$\|H({\bf x}^k)\|$ tends to $0$ as the number of iteration $k$
increases. And $\|\nabla \Psi({\bf x}^k)\|$ also tends to 0 except
that it increases from ${\bf k}=2$ to ${\bf k}=4$. This shows that
$\|\nabla \Psi({\bf x}^k)\|$ does converge to 0 but not converge
monotonically when the algorithm converges.

% For tables use
\begin{table}[h]
\centering
% table caption is above the table
\caption{Iterations of Algorithm \ref{alg4.1} for a random tensor
$\mathcal{A}\in S_{6,8}$ and corresponding ${\bf b}$}
\label{tab:1}       % Give a unique label
% For LaTeX tables use
\scalebox{0.9}{
\begin{tabular}{cccc}
\hline\noalign{\smallskip}
${\bf k}$  & ${\bf x}^k$ & $\|H({\bf x}^k)\|$& $\|\nabla \Psi({\bf x}^k)\|$ \\
\noalign{\smallskip}\hline\noalign{\smallskip}
0 & $($0.8143, 0.2435, 0.9293, 0.3500, 0.1966, 0.2511, 0.6160, 0.4733$)^\mathrm{T}$ & 562.2589  & 1500602.8826\\
1 & $($0.7407, 0.2435, 0.6880, 0.3545, 0.1072, 0.3670, 0.4555, 0.3271$)^\mathrm{T}$ & 148.1702 & 203193.6101\\
2 & $($0.4542, 0.4477, 0.4349, 0.4348 -0.2944, 0.7819, 0.4209, 0.3007$)^\mathrm{T}$ & 25.5263 & 23486.6536\\
3 & $($1.0757, 0.3147, 0.2655, 0.4343 -0.2368, 0.4908, 0.1690, 0.3781$)^\mathrm{T}$ & 20.2932 & 48494.4079\\
4 & $($1.2158, 0.3379, 0.4481, 0.5825 -0.2075, 0.1750, 0.3290, 0.0197$)^\mathrm{T}$ & 18.6526 & 49990.1630\\
5 & $($0.8865, 0.3812, 0.3176, 0.5179 -0.3075, 0.3684, 0.4486, 0.2840$)^\mathrm{T}$ & 10.0354 & 20912.3905\\
6 & $($0.8742, 0.2928, 0.3744, 0.5308 -0.3895, 0.5269, 0.2838, 0.3997$)^\mathrm{T}$ & 2.9292 & 3867.9206\\
7 & $($0.8798, 0.2888, 0.3406, 0.6301 -0.3722, 0.4799, 0.3198, 0.3293$)^\mathrm{T}$ & 1.3213 & 1522.0099\\
8 & $($0.8664, 0.2829, 0.3003, 0.6746 -0.3890, 0.4936, 0.3325, 0.3355$)^\mathrm{T}$ & 0.7455 & 1084.3075\\
9 & $($0.8684, 0.2850, 0.2737, 0.6914 -0.3985, 0.4960, 0.3394, 0.3411$)^\mathrm{T}$ & 0.1766 & 482.0095\\
10 & $($0.8690, 0.2852, 0.2752, 0.6895 -0.3976, 0.4957, 0.3383, 0.3411$)^\mathrm{T}$ & 0.0144 & 21.2907\\
11 & $($0.8692, 0.2853, 0.2753, 0.6894 -0.3975, 0.4956, 0.3383, 0.3410$)^\mathrm{T}$ & 0.0029 & 2.4370\\
12 & $($0.8692, 0.2853, 0.2754, 0.6893 -0.3975, 0.4956, 0.3383, 0.3410$)^\mathrm{T}$ & 0.0002 & 0.1396\\
13 & $($0.8692, 0.2853, 0.2754, 0.6892 -0.3975, 0.4956, 0.3383, 0.3410$)^\mathrm{T}$ & 0.0001 & 0.0008\\
14 & $($0.8692, 0.2853, 0.2754, 0.6892 -0.3975, 0.4956, 0.3383, 0.3410$)^\mathrm{T}$ & 0.0000 & 0.0000\\
\noalign{\smallskip}\hline
\end{tabular}}
\end{table}

The second numerical experiment aims to verify Theorem \ref{thm4}.
We first generate a random symmetric nonnegative tensor
$\mathcal{C}\in S(4,10)$ and a random vector ${\bf
z}^*=(0.1040,0.7455,0.7363,0.5619,0.1842,0.5972,0.2999,
0.1341,0.2126,$ $0.8949)^\mathrm{T}\in R^{10}$. All entries of
$\mathcal{C}$ are uniform random numbers in the interval $[0,1]$.
Let ${\bf b}=(\mathcal{C}-\mathcal{I}){{\bf z}^*}^{m-1}$. Since
$\mathcal{D}\in S(4,10)$ is a diagonal tensor whose diagonal
elements are $1$ or $-1$, there are at most $2^{10}=1024$ different
$\mathcal{D}$. The first attempt is to generate all these $1024$
tensors. For each tensor $\mathcal{D}_k$, set
$\mathcal{A}_k=\mathcal{C}\mathcal{D}_k$ (see Definition \ref{def2})
and ${\bf x_k}=(\mathcal{D}_k{\bf z^*}^{m-1})^{\frac{1}{m-1}}$. We
check whether $\mathcal{A}_k{\bf x_k}^{m-1}-|{\bf x_k}|^{m-1}$ is
equals to ${\bf b}$ for all $k\in\{1,2,...,1024\}$. The result shows
that each ${\bf x_k}$ is just one of the solution to the
corresponding TAVE problem $\mathcal{A}_k{\bf x}^{m-1}-|{\bf
x}|^{m-1}={\bf b}_k$.

The second attempt of the second numerical experiment is to generate
five $\mathcal{D}_k$ of all $1024$ tensors randomly and use
Algortihm \ref{alg4.1} to solve the corresponding TAVE. The diagonal
elements of the five $\mathcal{D}_k$ is shown in Table \ref{tab:2}.
\begin{table}[h]
\centering
% table caption is above the table
\caption{Diagonal elements of $\mathcal{D}_{k}$}
\label{tab:2}       % Give a unique label
% For LaTeX tables use
\begin{tabular}{cc}
\hline\noalign{\smallskip}
${\bf k}$  & diag of $\mathcal{D}_k$ \\
\noalign{\smallskip}\hline\noalign{\smallskip}
1 & $($-1,-1,-1,-1,-1,-1,-1,-1,-1,-1$)^\mathrm{T}$ \\
2 & $($-1, 1,-1, 1,-1, 1,-1,-1,-1, 1$)^\mathrm{T}$ \\
3 & $($ 1, 1,-1,-1, 1, 1,-1,-1,-1,-1$)^\mathrm{T}$ \\
4 & $($-1, 1,-1, 1,-1, 1,-1,-1, 1, 1$)^\mathrm{T}$ \\
5 & $($ 1,-1, 1, 1, 1, 1,-1, 1,-1, 1$)^\mathrm{T}$ \\
\noalign{\smallskip}\hline
\end{tabular}
\end{table}

We first select the initial points for Algorithm \ref{alg4.1} by
using normal distribution, i.e., entries are from standardized
normal distribution $N(0,1)$ independently. Here we call these
initial points {\em type-I} initial points. The results of
corresponding TAVE with {\em type-I} initial points is summarized in
Table \ref{tab:3}. We can easily find out that none of the five
${\bf x_k}$ is in the form of $(\mathcal{D}_k{\bf
z^*}^{m-1})^{\frac{1}{m-1}}$. Because Algorithm \ref{alg4.1} is
based on the thoughts of Newton method, thus its convergence relies
heavily on the initial point. In order to detect solution which is
mentioned in Theorem \ref{thm4} by Algorithm \ref{alg4.1}, we should
choose the initial points in another way. For each $\mathcal{D}_k$,
we generate {\em type-II} initial points by adding a random number
chosen from uniform distribution over $(-0.3,0.3)$ to
$(\mathcal{D}_k{\bf z^*}^{m-1})^{\frac{1}{m-1}}$. The results of
corresponding TAVE with {\em type-II} initial points is shown in
Table \ref{tab:4}. The solutions are exactly in the form of
$(\mathcal{D}_k{\bf z^*}^{m-1})^{\frac{1}{m-1}}$.

\begin{table}[h]
\centering
% table caption is above the table
\caption{Numerical results for tensors $\mathcal{A}_k$ with type-I
initial points}
\label{tab:3}       % Give a unique label
% For LaTeX tables use
\begin{tabular}{ccccc}
\hline\noalign{\smallskip}
${\bf k}$  & ${\bf x_k}$ & $\|H({\bf x_k})\|$ & \bf Iter. & \bf Time\\
\noalign{\smallskip}\hline\noalign{\smallskip} \multirow{2}{*}{1}
& $($-0.3485,-0.0971,-0.7753,-1.2447,-0.7739, & \multirow{2}{*}{0.00000012} & \multirow{2}{*}{15} & \multirow{2}{*}{0.2135}\\
& \quad\,-0.5628,-0.4868, 0.4480, 0.2925,-0.9003$)^\mathrm{T}$ & & &\\
\multirow{2}{*}{2}
& $($ 0.4184,-0.0423,-0.2989, 1.0357,-1.0340, & \multirow{2}{*}{0.00000022} & \multirow{2}{*}{15} & \multirow{2}{*}{0.2060}\\
& \quad\, 0.3109,-0.3686,-0.2755,-0.6852, 0.9528$)^\mathrm{T}$ & & &\\
\multirow{2}{*}{3}
& $($ 0.7454, 0.5055,-0.6641, 0.3093,-0.1769, & \multirow{2}{*}{0.00000003} & \multirow{2}{*}{18} & \multirow{2}{*}{0.2673}\\
& \quad\, 1.1273,-0.4514,-1.1430,-0.0619,-0.2421$)^\mathrm{T}$ & & &\\
\multirow{2}{*}{4}
& $($-0.9570, 0.5494,-2.1429,-0.1959,-1.8247, & \multirow{2}{*}{0.00000000} & \multirow{2}{*}{11} & \multirow{2}{*}{0.1355}\\
& \quad\,-0.3996, 0.8803,-0.3457, 0.0458, 0.1694$)^\mathrm{T}$ & & &\\
\multirow{2}{*}{5}
& $($ 0.3385,-1.1498, 1.0413, 0.3533, 0.7606, & \multirow{2}{*}{0.00000006} & \multirow{2}{*}{10} & \multirow{2}{*}{0.1265}\\
& \quad\,-0.1214,-0.3290,-0.0458,-0.2049, 0.4027$)^\mathrm{T}$ & & &\\
\noalign{\smallskip}\hline
\end{tabular}
\end{table}

\begin{table}[h]
\centering
% table caption is above the table
\caption{Numerical results for tensors $\mathcal{A}_k$  with type-II
initial points}
\label{tab:4}       % Give a unique label
% For LaTeX tables use
\begin{tabular}{ccccc}
\hline\noalign{\smallskip}
${\bf k}$  & ${\bf x_k}$ & $\|H({\bf x_k})\|$ & \bf Iter. & \bf Time\\
\noalign{\smallskip}\hline\noalign{\smallskip} \multirow{2}{*}{1}
& $($-0.1040,-0.7455,-0.7363,-0.5619,-0.1842, & \multirow{2}{*}{0.00000072} & \multirow{2}{*}{20} & \multirow{2}{*}{0.2523}\\
& \quad\,-0.5972,-0.2999,-0.1341,-0.2126,-0.8949$)^\mathrm{T}$ & & &\\
\multirow{2}{*}{2}
& $($-0.1040, 0.7455,-0.7363, 0.5619,-0.1842, & \multirow{2}{*}{0.00000090} & \multirow{2}{*}{17} & \multirow{2}{*}{0.2050}\\
& \quad\, 0.5972,-0.2999,-0.1341,-0.2126, 0.8949$)^\mathrm{T}$ & & &\\
\multirow{2}{*}{3}
& $($ 0.1040, 0.7455,-0.7363,-0.5619, 0.1842, & \multirow{2}{*}{0.00000091} & \multirow{2}{*}{24} & \multirow{2}{*}{0.2838}\\
& \quad\, 0.5972,-0.2999,-0.1341,-0.2126,-0.8949$)^\mathrm{T}$ & & &\\
\multirow{2}{*}{4}
& $($-0.1040, 0.7455,-0.7363, 0.5619,-0.1842, & \multirow{2}{*}{0.00000064} & \multirow{2}{*}{16} & \multirow{2}{*}{0.1896}\\
& \quad\, 0.5972,-0.2999,-0.1341, 0.2126, 0.8949$)^\mathrm{T}$ & & &\\
\multirow{2}{*}{5}
& $($ 0.1040,-0.7455, 0.7363, 0.5619, 0.1842, & \multirow{2}{*}{0.00000075} & \multirow{2}{*}{14} & \multirow{2}{*}{0.1638}\\
& \quad\, 0.5972,-0.2999, 0.1341,-0.2126, 0.8949$)^\mathrm{T}$ & & &\\
\noalign{\smallskip}\hline
\end{tabular}
\end{table}

In Tables \ref{tab:3} and \ref{tab:4}, ${\bf k}$ denotes the
experiment No. corresponding to Table \ref{tab:2}. ${\bf x_k}$
denotes the solution vectors returned by Algorithm \ref{alg4.1}.
$\|H({\bf x_k})\|$ denotes the Euclid norm of $H({\bf x_k})$. If the
norm of $H({\bf x_k})$ is small enough, we can regard ${\bf x_k}$ as
an approximate solution of TAVE. {\bf Iter.} denotes the number of
iteration and {\bf Time} denotes the time of iteration that finds
corresponding $x$ by Algorithm \ref{alg4.1}. In the second
experiment, we verify Theorem \ref{thm4} from the instant correctly.
Besides, from Table \ref{tab:3} and \ref{tab:4}, we find that under
the conditions of Theorem \ref{thm4}, the solution $(\mathcal{D}{\bf
z^*}^{m-1})^{\frac{1}{m-1}}$ may not be the only solution of TAVE
$\mathcal{A}{\bf x}^{m-1}-|{\bf x}|^{m-1}={\bf b}$. There might be
some other solutions, such as the solution in Table \ref{tab:3}. To
discuss the uniqueness of the positive solution, we conduct our
third experiment.

Our third numerical experiment focuses on Theorem \ref{thm2}. Here
we first generate a random symmetric nonnegative tensor
$\mathcal{B}$ whose entries are uniform random numbers in the
interval $[0,1]$. Let $c=1+(1+0.01)\max_{1\leq i\leq
n}{(\mathcal{B}e^3)_i}$, where $e=(1,1,1,1)^\mathrm{T}$. Since
$\max_{1\leq i\leq n}{(\mathcal{B}e^3)_i}\geq\rho(\mathcal{B})$, the
choice of $c$ makes sure that $c>\rho(\mathcal{B})+1$. Then let
$\mathcal{A}=c\mathcal{I}-\mathcal{B}$, and $\mathcal{A}$ satisfies
the conditions of Theorem \ref{thm2}, i.e.,
$\mathcal{A}-\mathcal{I}$ is strong M-tensor. Tensors $\mathcal{B}$
and $\mathcal{A}$ are given in Tables \ref{tab:5} and \ref{tab:6},
respectively.
\begin{table}[h]
\centering
% table caption is above the table
\caption{A random symmetric nonnegative tensor
$\mathcal{B}=(b_{i_1i_2i_3i_4})\in S(4,4)$}
\label{tab:5}       % Give a unique label
% For LaTeX tables use
\begin{tabular}{ccccc}
%\hline\noalign{\smallskip}
%$k$  & diag of $\mathcal{D}_k$ \\
\noalign{\smallskip}\hline\noalign{\smallskip}
$b_{1111}=0.8147$ & $b_{1112}=0.9058$ & $b_{1113}=0.1270$ & $b_{1114}=0.9134$ & $b_{1122}=0.6324$ \\
$b_{1123}=0.0975$ & $b_{1124}=0.2785$ & $b_{1133}=0.5469$ & $b_{1134}=0.9575$ & $b_{1144}=0.9649$ \\
$b_{1222}=0.1576$ & $b_{1223}=0.9706$ & $b_{1224}=0.9572$ & $b_{1233}=0.4854$ & $b_{1234}=0.8003$ \\
$b_{1244}=0.1419$ & $b_{1333}=0.4218$ & $b_{1334}=0.9157$ & $b_{1344}=0.7922$ & $b_{1444}=0.9595$ \\
$b_{2222}=0.6557$ & $b_{2223}=0.0357$ & $b_{2224}=0.8491$ & $b_{2233}=0.9340$ & $b_{2234}=0.6787$ \\
$b_{2244}=0.7577$ & $b_{2333}=0.7431$ & $b_{2334}=0.3922$ & $b_{2344}=0.6555$ & $b_{2444}=0.1712$ \\
$b_{3333}=0.7060$ & $b_{3334}=0.0318$ & $b_{3344}=0.2769$ & $b_{3444}=0.0462$ & $b_{4444}=0.0971$ \\
\noalign{\smallskip}\hline
\end{tabular}
\end{table}

\begin{table}[h]
\centering
 \caption{The symmetric tensor
$\mathcal{A}=(a_{i_1i_2i_3i_4})\in S(4,4)$ based on $\mathcal{B}$}
\label{tab:6}
% Give a unique label
% For LaTeX tables use
\scalebox{0.9}{
\begin{tabular}{ccccc}
%\hline\noalign{\smallskip}
%$k$  & diag of $\mathcal{D}_k$ \\
\noalign{\smallskip}\hline\noalign{\smallskip}
$a_{1111}=\,\,40.8037$ & $a_{1112}=-0.9058$ & $a_{1113}=-0.1270$ & $a_{1114}=-0.9134$ & $a_{1122}=-0.6324$ \\
$a_{1123}=-0.0975$ & $a_{1124}=-0.2785$ & $a_{1133}=-0.5469$ & $a_{1134}=-0.9575$ & $a_{1144}=-0.9649$ \\
$a_{1222}=-0.1576$ & $a_{1223}=-0.9706$ & $a_{1224}=-0.9572$ & $a_{1233}=-0.4854$ & $a_{1234}=-0.8003$ \\
$a_{1244}=-0.1419$ & $a_{1333}=-0.4218$ & $a_{1334}=-0.9157$ & $a_{1344}=-0.7922$ & $a_{1444}=-0.9595$ \\
$a_{2222}=\,\,40.9627$ & $a_{2223}=-0.0357$ & $a_{2224}=-0.8491$ & $a_{2233}=-0.9340$ & $a_{2234}=-0.6787$ \\
$a_{2244}=-0.7577$ & $a_{2333}=-0.7431$ & $a_{2334}=-0.3922$ & $a_{2344}=-0.6555$ & $a_{2444}=-0.1712$ \\
$a_{3333}=\,\,40.9124$ & $a_{3334}=-0.0318$ & $a_{3344}=-0.2769$ & $a_{3444}=-0.0462$ & $a_{4444}=\,\,41.5213$\\
\noalign{\smallskip}\hline
\end{tabular}}
\end{table}

We choose $10$ random positive vectors ${\bf b}_k\in
R_+^4,k=1,\ldots,10$. For each ${\bf b}_k$, we find $20$ repeatable
solutions of TAVE: $\mathcal{A}{\bf x}^{m-1}-|{\bf x}|^{m-1}={\bf
b}_k$ with random vectors from $N(0,1)^4$ as initial points
repeatedly and summarize the results in Table \ref{tab:7}.

\begin{table}[h]
\centering \caption{Numerical results for the third experiment}
\label{tab:7}
% Give a unique label
% For LaTeX tables use
\scalebox{0.78}{
\begin{tabular}{cccccc}
\hline\noalign{\smallskip}
${\bf x}$ & ${\bf b}$ & {\bf Iter.} & {\bf Time} & $\max{\|H({\bf x})\|}$ & {\bf Attempts}\\
\noalign{\smallskip}\hline\noalign{\smallskip}
$(0.8100,0.7881,0.7786,0.8003)^\mathrm{T}$ & $(1.4193,0.2916,0.1978,1.5877)^\mathrm{T}$ & 31.00 & 0.6783 & 0.00000098 & 20/100\\
$(0.7285,0.7212,0.7156,0.7098)^\mathrm{T}$ & $(0.8045,0.6966,0.8351,0.2437)^\mathrm{T}$ & 19.40 & 0.3109 & 0.00000099 & 20/157\\
$(0.7219,0.7313,0.7230,0.7098)^\mathrm{T}$ & $(0.2157,1.1658,1.1480,0.1049)^\mathrm{T}$ & 19.55 & 0.3456 & 0.00000099 & 20/205\\
$(0.8453,0.8603,0.8294,0.8276)^\mathrm{T}$ & $(0.7223,2.5855,0.6669,0.1873)^\mathrm{T}$ & 13.65 & 0.1907 & 0.00000082 & 20/244\\
$(0.8445,0.8584,0.8321,0.8507)^\mathrm{T}$ & $(0.0825,1.9330,0.4390,1.7947)^\mathrm{T}$ & 14.05 & 0.2168 & 0.00000084 & 20/290\\
$(0.7104,0.7055,0.6849,0.6957)^\mathrm{T}$ & $(0.8404,0.8880,0.1001,0.5445)^\mathrm{T}$ & 68.25 & 1.7492 & 0.00000051 & 20/145\\
$(0.6775,0.6771,0.6677,0.6750)^\mathrm{T}$ & $(0.3035,0.6003,0.4900,0.7394)^\mathrm{T}$ & 21.75 & 0.3864 & 0.00000099 & 20/216\\
$(0.9021,0.8787,0.8894,0.8805)^\mathrm{T}$ & $(1.7119,0.1941,2.1384,0.8396)^\mathrm{T}$ & 15.70 & 0.2535 & 0.00000089 & 20/114\\
$(0.8104,0.8007,0.7908,0.7841)^\mathrm{T}$ & $(1.3546,1.0722,0.9610,0.1240)^\mathrm{T}$ & 14.60 & 0.2121 & 0.00000071 & 20/129\\
$(0.8957,0.8939,0.8661,0.8808)^\mathrm{T}$ & $(1.4367,1.9609,0.1977,1.2078)^\mathrm{T}$ & 13.60 & 0.1980 & 0.00000099 & 20/114\\
\noalign{\smallskip}\hline
\end{tabular}}
\end{table}

In Table \ref{tab:7}, ${\bf x}$ denotes the solution of TAVE. ${\bf
b}$ denotes random generated ${\bf b}_k$ mentioned above. {\bf
Iter.} denotes the average number of iteration that finds the
corresponding solution successfully. {\bf Time} denotes the average
time of iteration that finds the corresponding solution by Algorithm
1. $\max{\|H({\bf x})\|}$ is the maximum norm of all $H({\bf x})$
returned by Algorithm 1 whose ${\bf x}$ is the corresponding
solution. {\bf Attempts} has the form {\bf N/T}, {\bf N} denotes the
number of the corresponding ${\bf x}$ found by Algorithm
\ref{alg4.1} and {\bf T} denotes the number of initial points in
all.

In this experiment, we use ``while'' loop in Matlab program to
guarantee that we can get exact $20$ solutions (might be repeatable)
for each $\bf b_k\geq 0$. According to Table \ref{tab:7}, for each
$\bf b_k\geq 0$, Algorithm \ref{alg4.1} only returns unique positive
solution in all $20$ repeatable solutions. This phenomenon fits
Theorem \ref{thm2} very well. Besides, in order to get $20$ valid
solutions, the initial points we attempt is $10$ times more than the
valid ones. This means that most of the random initial points fail
to find a solution by Algorithm \ref{alg4.1}. The reason might be
that the convergence of Newton type method depends on the initial
point badly. Theorem \ref{thm2} shows that under this circumstances,
there's only one unique positive solution of TAVE. If and only if
the initial point is in the convergence region of some solution of
TAVE, the algorithm will converge. Therefore, it's harder to find
valid solutions if $\bf b\geq 0$. Table \ref{tab:8} shows the
solutions found by Algorithm \ref{alg4.1} when ${\bf b}=(-1,1,1,1)$.
The initial points attempted in all is much less.

\begin{table}[h]
\centering \caption{Solutions of TAVE when ${\bf b}=(-1,1,1,1)$}
\label{tab:8}
% Give a unique label
% For LaTeX tables use
\scalebox{0.88}{
\begin{tabular}{cccccc}
\hline\noalign{\smallskip}
${\bf x}$ & ${\bf b}$ & {\bf Iter.} & {\bf Time} & $\max{\|H({\bf x})\|}$ & {\bf Attempts}\\
\noalign{\smallskip}\hline\noalign{\smallskip}
$(\,\,\,\,\, 0.0800,0.3629,0.3543,0.3505)^\mathrm{T}$ & $(-1,1,1,1)^\mathrm{T}$ & 12.67 & 0.1644 & 0.00000070 & \,\,\,3/20\\
$(-0.2593,0.2948,0.2891,0.2903)^\mathrm{T}$ & $(-1,1,1,1)^\mathrm{T}$ & 13.67 & 0.1708 & 0.00000099 & \,\,\,3/20\\
$(\,\,\,\,\, 0.6258,0.6600,0.6522,0.6537)^\mathrm{T}$ & $(-1,1,1,1)^\mathrm{T}$ & 11.93 & 0.1516 & 0.00000075 & 14/20\\
\noalign{\smallskip}\hline
\end{tabular}}
\end{table}

Moreover, under the circumstances that $\bf b\geq 0$ and
$\mathcal{A}-\mathcal{I}$ is strong M-tensor, whether the unique
positive solution of TAVE is the unique solution of TAVE remains a
question. In our experiment we haven't found other solutions except
for the unique positive ones.

\section{Conclusion}
We have introduced tensor absolute value equations. The simple
definition is a natural generalization of the definition of absolute
value equations in the matrix case. We have established some basic
properties for tensor absolute value equations and we reformulate
tensor absolute value equations  as a generalized tensor
complementarity problem. We have proposed some sufficient conditions
for the existence of solution to the multilinear equations. We
propose an inexact Levenberg-Marquardt-type method (Algorithm
\ref{alg4.1}) to solve the tensor absolute value equations and some
numerical results have shown that our algorithm is performing well.

There are some questions which are still in study. For example, we
known that ``The AVE (\ref{ave}) is uniquely solvable for any ${\bf
b}\in R^n$ if the singular values of $A$ exceed $1$''
\cite{mangas3}. Can we extend the conclusion to TVAE (\ref{TAVE}),
i.e.,  the statement ``The TAVE (\ref{TAVE}) is uniquely solvable
for any ${\bf b}\in R^n$ if the singular values of tensor
$\mathcal{A}$ exceed $1$'' is correct or not? This is still an open
question.

\section*{Acknowledgments}
%The authors would like to thank the editor and the anonymous
%referees for their constructive comments and suggestions which lead
%to a significantly improved version of the paper.

Shouqiang Du's work was supported by the National Natural Science
Foundation of China (Grant No. 11671220, 11401331) and the Nature
Science Foundation of Shandong Province (ZR2015AQ013, ZR2016AM29).
Liping Zhang's  work was supported by the National Natural Science
Foundation of China (Grant No. 11271221). Liqun Qi's work was
supported by the Hong Kong Research Grant Council (Grant No. PolyU
501212, 501913, 15302114 and 15300715).


\begin{thebibliography}{}


\bibitem{bader2}
{\sc B.W.~Bader, T.G.~Kolda}, {\em et al.}, {\em MATLAB Tensor
Toolbox Version 2.6} (2012). http://www.sandia.gov/$\sim$
tgkolda/TensorToolbox/

\bibitem{bon3}
{\sc J.F.~Bonnans, R.~Cominetti, and A.~Shapiro}, {\em Second order optimality conditions based on parabolic second order tangent sets}, SIAM Journal on Optimization, 9 (1999), pp.~466--493.

\bibitem{ZGQ}
{\sc K.~Chang, L.~Qi, and T.~Zhang}, {\em A survey on the spectral theory of nonnegative tensors}, Numerical Linear
Algebra with Applications, 20 (2013), pp.~891--912.

\bibitem{cheqiwei}
{\sc M.~Che, L.~Qi, and Y.~Wei}, {\em Positive definite tensors to nonlinear
complementarity problems}, Journal of Optimization Theory and
Applications, 168 (2016), pp.~475--487.

\bibitem{chenqi}
{\sc M.~Che and L.~Qi}, {\em A semismooth Newton method for tensor eigenvalue complementarity problem}, Computational Optimization and Applications,
65 (2016), pp.~109--126.


\bibitem{clark9}
{\sc F.H.~Clarke}, {\em Optimization and Nonsmooth Analysis}, Wiley, New
York, 1983.


\bibitem{LCP}
{\sc R.W.~Cottle, J.-S.~Pang, and R.E.~Stone}, {\em The Linear
Complementarity Problem}, Academic Press, Boston, 1992.

\bibitem{dingwei}
{\sc W.~Ding and Y.~Wei}, {\em Solving multi-linear systems with
$M$-Tensors}, Journal of Scientific Computing, 68 (2016),
pp.~683--715.

\bibitem{dingqiwei}
{\sc W.~Ding, L.~Qi, and Y.~Wei}, {\em M-tensors and nonsingular
M-tensors}, Linear Algebra and Its Applications, 439 (2013),
pp.~3264--3278.

\bibitem{NCP}
{\sc F.~Facchinei and J.-S.~Pang}, {\em Finite-Dimensional
Variational Inequalities and Complementarity Problems},
Springer-Verlag, New York, 2003.

\bibitem{FK} {\sc F.~Facchinei and C.~Kanzow},  {\em A nonsmooth inexact Newton method for the solution of large-scale nonlinear
complementarity problems}, Mathematical Programming, 76 (1997),
pp.~493--512.

\bibitem{fbf}
{\sc A.~Fischer}, {\em A special Newton-type optimization method},
Optimization, 24 (1992), pp.~269--284.

\bibitem{hin17} {\sc M.~Hinterm\"{u}ller}, {\em Semismooth Newton methods and applications}, Department of Mathematics, Humboldt-University of Berlin,
2010.

\bibitem{huangqi}
{\sc Z.~Huang and L.~Qi}, {\em Formulating an $n$-person
noncooperative game as a tensor complementarity problem},
Computational Optimization and Applications (2016),
DOI:10.1007/s10589-016-9872-7

\bibitem{ling}
{\sc X.~Li and M.K.~Ng}, {\em Solving sparse non-negative tensor
equations: algorithms and applications}, Frontiers of Mathematics in
China, 10 (2015), pp.~649--680.

\bibitem{lim}
{\sc L.-H.~Lim}, {\em Singular values and eigenvalues of tensors: A
variational approach}, in {\rm IEEE CAMSAP} 2005: First
International Workshop on Computational Advances in Multi-Sensor
Adaptive Processing, 2005, pp.~129--132.

\bibitem{mangas1}{\sc O.L.~Mangasarian}, {\em Absolute value programming}, Computational  Optimization and  Applications,
 36 (2007), pp.~43--53.

\bibitem{mangas2} \sameauthor, {\em Knapsack feasibility as an value equation
   solvable by successive linear programming}, Optimization Letters, 3 (2009),
   pp.~161--170.

\bibitem{mangas3}{\sc O.L.~Mangasarian and R.R.~Meyer},  {\em Absolute value equations}, Linear Algebra and its Applications, 419 (2006),
pp.~359--367.

\bibitem{miff26} {\sc R.~Mifflin}, {\em Semismooth and semiconvex functions in constrained optimization}, SIAM Journal on Control and Optimization, 15 (1977),
pp.~959--972.

\bibitem{qi05}{\sc L.~Qi}, {\em Eigenvalues of a real supersymmetric tensor}, Journal of Symbolic Computation, 40 (2005),
pp.~1302--1324.

\bibitem{qi32} {\sc L.~Qi and J.~Sun}, {\em A nonsmooth version of Newton's method}, Mathematical Programming, 58 (1993),
pp.~353--367.

\bibitem{qi33} {\sc L.~Qi and H.~Yin}, {\em A strongly semismooth integral function and its application}, Computation Optimization and Applications, 25 (2003),
pp.~223--246.

\bibitem{qisunzhou}
{\sc L.~Qi, D.~Sun, and G.~Zhou}, {\em A new look at smoothing
Newton methods for nonlinear complementarity problems and box
constrained variational inequalities}, Mathematical Programming, 87
(2000), pp.~1--35.

\bibitem{songqi}
{\sc Y.~Song and L.~Qi}, {\em Properties of some classes of
structured tensors}, Journal of Optimization Theory and
Applications, 165 (2015), pp.~854--873.

\bibitem{songqi2}\sameauthor, {\em Tensor complementarity problem and semi-positive
tensors}, Journal of Optimization Theory and Applications, 169
(2016), pp.~1069--1078.


\bibitem{minf}{\sc D.~Sun and J.~Sun}, {\em Strong semismoothness of the Fischer-Burmeister SDC and SOC complementarity
functions}, Mathematical Programming, 103 (2005), pp.~575--581.

\bibitem{sunqi37}{\sc D.~Sun and L.~Qi}, {\em On NCP-functions}, Computation Optimization and Applications, 13 (1999),
pp.~201--220.

\bibitem{zhangqizhou}{\sc L.~Zhang, L.~Qi, and G.~Zhou}, {\em M-tensors and some applications}, SIAM Journal on Matrix Analysis and Applications, 35 (2014),
pp.~437--452.

\bibitem{zhou}{\sc G.~Zhou, L.~Caccetta, and K.L.~Teo}, {\em A superlinearly convergent method for a class of complementarity
problems with non-Lipschitzian functions}, SIAM Journal on
Optimization, 20 (2010), pp.~1811--1827.

\end{thebibliography}
\end{document}